\documentclass[12pt,centertags]{amsart}
\usepackage{amsmath,amstext,amsthm,amscd}
\usepackage{amssymb}
\usepackage[mathscr]{eucal}
\usepackage{mathrsfs}

\newcommand{\field}[1]{\mathbb{#1}}
\newcommand{\Z}{\field{Z}}

\newcommand{\C}{\field{C}}
\newcommand{\N}{\field{N}}

\def\cC{\mathscr{C}}

\def\cK{\mathscr{K}}

\def\mO{\mathcal{O}}

\DeclareMathOperator{\orb}{orb}

\DeclareMathOperator{\sing}{sing}

\newcommand\mS{\mathcal{S}}
\newcommand{\om}{\omega}

\def\Proclaim #1. #2\par{\bigbreak\noindent{\sc#1.\enspace}{\it#2}\par}
\newtheorem{lem}{Lemma}[section]

\newtheorem{thm}[lem]{Theorem}

\newtheorem{rem}[lem]{Remark}

\newcommand{\be}{\begin{eqnarray}}
\newcommand{\ee}{\end{eqnarray}}
\newcommand{\ov}{\overline}

\newcommand{\wi}{\widetilde}
\newcommand{\var}{\varepsilon}

\newcommand{\comment}[1]{}

\begin{document}

\title{A remark on weighted Bergman kernels on orbifolds}



\author{Xianzhe Dai}
\address{Department of Mathematics, UCSB, CA 93106 USA}
\email{dai@math.ucsb.edu}
\author{Kefeng Liu}
\address{Center of Mathematical Science, Zhejiang University
and Department of Mathematics, UCLA, CA 90095-1555,
USA}
\email{liu@math.ucla.edu}
\author{Xiaonan Ma}
\address{Universit{\'e} Paris Diderot - Paris 7,
UFR de Math{\'e}matiques, Case 7012, Site Chevaleret, 75205 Paris
Cedex 13, France} \email{ma@math.jussieu.fr}

\begin{abstract}
 In
this note, we explain that Ross-Thomas' result \cite[Theorem
1.7]{RT09a} on the weighted Bergman kernels on orbifolds can be directly deduced
from our previous result \cite{DLM06}. This result
plays an important role in the companion paper \cite{RT09b} to prove
an orbifold version of Donaldson Theorem. 
\end{abstract}
\maketitle

In two very interesting papers \cite{RT09a, RT09b}, Ross-Thomas
describe a notion of ampleness for line bundles on K\"ahler
orbifolds with cyclic quotient singularities which is related to
embeddings in weighted projective spaces. They then apply the
results in \cite{RT09a} to prove an orbifold version of Donaldson
Theorem \cite{RT09b}. Namely, the existence of an orbifold K\"ahler
metric with constant scalar curvature implies certain stability
condition for the orbifold. In these papers, the result
\cite[Theorem 1.7]{RT09a} on the asymptotic expansion of Bergman
kernels plays a crucial role.

In this note, we explain how to directly derive Ross-Thomas' result
\cite[Theorem 1.7]{RT09a} from Dai-Liu-Ma \cite[(5.25)]{DLM06}, 
provided Ross-Thomas condition \cite[(1.8)]{RT09a} on
$c_{i}$ holds. Since in \cite[\S 5]{DLM06}, we state our results for general
symplectic orbifolds, in what follows, we will just use the version
from \cite[Theorem 5.4.11]{MM07}, where Ma-Marinescu wrote them in detail
for K\"ahler orbifolds.   We will use freely the notation in
\cite[\S 5.4]{MM07}. We assume also the auxiliary vector bundle $E$
therein is $\C$.

Let $(X,J, \om)$ be a compact $n$-dimensional K\"ahler orbifold
 with complex structure $J$, and with singular set $X_{\sing}$.
Let $(L,h^L)$ be a holomorphic Hermitian proper orbifold line bundle
on $X$. Let $\nabla ^L$ be the holomorphic Hermitian
connections on $(L,h^L)$  with curvature
$R^L=(\nabla^L)^2$.

We assume that $(L,h^L,\nabla^L)$ is a prequantum line bundle, i.e.,
\begin{align} \label{toe2.1}
 R^L =- 2 \pi\sqrt{-1}\, \om.
\end{align}


Let $g^{TX}=\om(\cdot,J\cdot)$ be the Riemannian metric on $X$
induced by $\om$. 
Let $\nabla^{TX}$ be the Levi-Civita connection on $(X, g^{TX})$.
We denote by
$R^{TX}=(\nabla^{TX})^2$ the curvature, 
by $r^{X}$ the scalar curvature of $\nabla^{TX}$.
For $x\in X$, set $d(x,X_{\sing}):=\inf_{y\in X_{\sing}} d(x,y)$
the distance from $x$ to $X_{\sing}$.

For $p\in \N$, the Bergman kernel $P_p(x,x')$ ($x,x'\in X$)
is the smooth kernel of the orthogonal projection
from $\cC^\infty(X, L^p)$ onto $H^0(X,L^p)$,
with respect to the Riemannian volume form $dv_X(x')$.

\begin{thm}[{\cite[Theorem 1.4]{DLM06}}, {\cite[Theorem 5.4.10]{MM07}}] \label{pbt4.12}
   There exist smooth coefficients
$\pmb{b}_r(x)\in \cC^{\infty}(X)$ which are polynomials in $R^{TX}$,
and its derivatives with order $\leqslant 2r-2$ at $x$ and $C_0>0$
such that for any $k,l\in \N$, there exist
$C_{k,l}>0$, $M\in \N$ with
\begin{multline}\label{pb4.20}
\Big |\frac{1}{p^n}P_p (x,x)
- \sum_{r=0}^{k} \pmb{b}_r(x) p^{-r} \Big |_{\cC^l} \\
\leqslant C_{k,l}
\Big (p^{-k-1} + p^{l/2}(1+\sqrt{p}d(x,X_{\sing}) )^M
e^{- \sqrt{C_0 p}\, d(x,X_{\sing})}\Big )\,,
\end{multline}
for any $x\in X$, $p\in \N^*$.
Moreover 
\be\label{bk2.7}
\pmb{b}_0=1, \quad \pmb{b}_1 = \frac{1}{8\pi} r^X .
\ee
\end{thm}

In local coordinates, there is a more precise form
{\cite[(5.25)]{DLM06}}, see also {\cite[Theorem 5.4.11]{MM07}}. Let
$\{x_i\}_{i=1}^I\subset X_{\sing}$. For each point $x_i$ we consider 
corresponding local charts $(G_{x_i},\wi{U}_{x_i})\to U_{x_i}$ with
$\wi{U}_{x_i}\subset \C^n$, such that $0\in \wi{U}_{x_i}$ is the inverse
image of $x_i\in U_{x_i}$, and $0$ is a fixed point of the finite stabilizer group
$G_{x_i}$ at $x_i$,
which acts $\C$-linearly and effectively on $\C^n$ (cf.
\cite[Lemma 5.4.3]{MM07}). We assume moreover that
$$B^{\wi{U}_{x_i}}(0,2\varepsilon)\subset \wi{U}_{x_i},
\mbox{ and } X_{\sing}\subset W:=\cup_{i=1}^I
B^{\wi{U}_{x_i}}(0,\frac{1}{4}\varepsilon)/G_{x_i}.$$ Let
$\wi{U}_{x_i}^g$ be the fixed point set of $g\in G_{x_i}$ in
$\wi{U}_{x_i}$, and let $\wi{N}_{x_i,g}$ be the normal bundle of
$\wi{U}_{x_i}^g$ in $\wi{U}_{x_i}$. For each $g\in G_{x_i}$, the
exponential map $\wi{N}_{x_i,g,\wi{x}}\ni Y\to
\exp^{\wi{U}_{x_i}}_{\wi{x}}(Y)$ identifies a neighborhood of
$\wi{U}_{x_i}^g$ with $\wi{W}_{x_i,g}=\{Y\in \wi{N}_{x_i,g},
|Y|\leqslant \varepsilon\}$. We identify $L|_{\wi{W}_{x_i,g}}$ with
$L|_{\wi{U}_{x_i}^g}$ by using the parallel transport along the
above exponential map. Then the $g$-action on $L|_{\wi{W}_{x_i,g}}$
is the multiplication by $e^{i\theta_g}$, and $\theta_g$ is locally
constant on $\wi{U}_{x_i}^g$.


Let $\nabla^{\wi{N}_{x_i,g}}$ be the connection on $\wi{N}_{x_i,g}$
induced by the Levi-Civita connection via projection.
We trivialize $\wi{N}_{x_i,g} \simeq \wi{U}_{x_i}^g\times \C^{n_g}$
by the parallel transport along the curve
$[0,1]\ni t\to t\wi{Z}_{1,g}$ for $\wi{Z}_{1,g}\in \wi{U}_{x_i}^g$,
which identifies also the metric on $\wi{N}_{x_i,g}$ with the canonical
metric on $\C^{n_g}$.
If $\wi{Z}\in \wi{W}_{x_i,g}$,
we will write $\wi{Z}= (\wi{Z}_{1,g}, \wi{Z}_{2,g})$ with
$\wi{Z}_{1,g}\in \wi{U}_{x_i}^g$, $\wi{Z}_{2,g}\in \C^{n_g}$. We will denote by
$Z$ the corresponding point on the orbifold.

\begin{thm}[{\cite[(5.25)]{DLM06}}, {\cite[Theorem 5.4.11]{MM07}}]
\label{pbt4.120}   On $\wi{U}_{x_i}$ as above,
there exist polynomials $\cK_{r,\wi{Z}_{1,g}}(\wi{Z}_{2,g})$ in $\wi{Z}_{2,g}$
of degree $\leqslant 3r$, of the same parity as $r$, whose coefficients
are polynomials in $R^{TX}$ and its derivatives of
order $\leqslant r-2$, and a constant $C_0>0$ such that for any $k,l\in \N$,
there exist $C_{k,l}>0$, $N\in \N$ such that
\begin{multline}\label{apb4.20}
\left |\frac{1}{p^n} P_p(\wi{Z},\wi{Z})
- \sum_{r=0}^k \pmb{b}_r(\wi{Z})p^{-r}\right.\\
\left. - \sum_{r=0}^{2k}  p^{-\frac{r}{2}}
\sum_{1\neq g\in G_{x_0}} e^{i \theta_g p}
\cK_{r,\wi{Z}_{1,g}}(\sqrt{p}\wi{Z}_{2,g}) e^{-2\pi
p\langle(1-g^{-1})\wi{z}_{2,g}, \ov{\wi{z}_{2,g}}\rangle}
\right |_{\cC^{l}}\\
\leqslant C_{k,l} \left(p^{-k-1} + p^{-k+\frac{l-1}{2}}
\left(1+\sqrt{p}d (Z,X_{\sing})\right)^N
e^{- \sqrt{C_0\, p}\, d (Z,X_{\sing})}\right),
\end{multline}
for any $|\wi{Z}|\leqslant \var/2$, $p\in \N$, with $\pmb{b}_r(\wi{Z})$ as in
Theorem \ref{pbt4.12} and
$\cK_{0,\wi{Z}_{1,g}}= 1$.
\end{thm}



Given a function $f(p,x)$ in $p\in \N$ and $x\in X$, we write
$f=\mO_{\cC^{j}}(p^{l})$ if the $\cC^{j}$-norm of $f$ is uniformly
bounded by $C\,p^{l}$.

\begin{thm}\label{pbt4.13} Let $(X,\om)$ be a compact $n$-dimensional
K\"ahler orbifold with cyclic quotient singularities
(i.e., the stabilizer group $G_{x}$ is a cyclic group for any $x\in X$),
and $L$ be a proper orbifold line bundle on $X$ equipped 
with a Hermitian metric $h^L$
   whose curvature form is $-2\pi \sqrt{-1}\, \om$, such that for any $x\in
   X$, the stabilizer group $G_{x}$ acts on $L_{\wi{x}}$ as
   $\Z_{|G_{x}|}$-order cyclic group.
Fix $ N\geq 0$, and $r\geq 0$ and suppose $c_{i}$ are a finite
number of positive constants chosen so that
if $X$ has an orbifold point of order $m$ then
\begin{align}\label{ross1.1}
\frac{1}{m}\sum_{i} i^k c_{i}= \sum_{i\equiv u \, mod \,m} i^{k}c_{i}
\quad \text{for all $u$ and all  } k=0, \cdots, N+r.
\end{align}
Then the function
\begin{align}\label{ross1.2}
B^{\orb}_{p}(x): = \sum_{i} c_{i} P_{p+i}(x,x).
\end{align}
admits a global $\cC^{2r}$-expansion of order $N$. That is,
there exist smooth functions
$b_{0},\cdots, b_{N}$ on $X$ such that
\begin{align}\label{ross1.4}
   B^{\orb}_{p}= \sum_{j=0}^{N} b_{j} p^{n-j} +  \mO_{\cC^{2r}}(p^{n-N-1}).
\end{align}
Furthermore, $b_{j}$ are universal polynomials in the constants
$c_{i}$ and the derivatives of $\om$; in particular
\begin{align}\label{ross1.5}
b_{0}=\sum_{i}c_{i},\quad b_{1}=\sum_{i} c_{i}\Big(n\,  i
+ \frac{1}{8\pi}r^{X}\Big).
\end{align}
\end{thm}

\begin{rem}\label{pbt4.14}
Theorem \ref{pbt4.13} recovers
\cite[Theorem 1.7]{RT09a} of Ross-Thomas, where 
the remainder estimate is $\mO_{\cC^{r}}(p^{n-N-1})$.

We improve here their remainder
estimate to $\mO_{\cC^{2r}}(p^{n-N-1})$ and we get Theorem
\ref{pbt4.13} directly from Theorems \ref{pbt4.12}, \ref{pbt4.120}.
\end{rem}

\begin{rem}\label{pbt4.15}
By Ma-Marinescu \cite[(3.30), Remark
3.10]{MM08a}, \cite[Theorem 4.1.3, Remark 5.4.13]{MM07},
Theorem \ref{pbt4.13} generalizes to
any $J$-invariant metric $g^{TX}$ on $TX$. Set $\Theta:= g^{TX}(J\cdot, \cdot)$.
The only change is that the
coefficients in the expansion become
\begin{align}\label{eq:}\begin{split}
&b_{0}= \frac{\om^{n}}{\Theta^{n}}\sum_{i}c_{i},\, \,
b_{1}=  \frac{\om^{n}}{\Theta^{n}}\sum_{i}c_{i}\left[
n\, i +\frac{r^{X}_{\om}}{8\pi} 
 -\frac{1}{4\pi} \Delta_{\om} \log\left(\frac{\om^{n}}{\Theta^{n}}\right)\right],
\end{split}\end{align}
where  $r^X_\om$, $\Delta_\om$ are the scalar curvature and
the Bochner Laplacian associated  to
$g^{TX}_{\om}= \om(\cdot,J\cdot)$.
Moreover, \eqref{ross1.4} can be taken to be uniform as
$(h^{L}, g^{TX})$ runs
over a compact set.
\end{rem}


\begin{proof}[Proof of Theorem \ref{pbt4.13}]
Recall that now $G$ is a cyclic group of order $m$.
Let $\zeta$ be a generator of $G$. From the local condition for orbi-ample line bundles,
$\zeta$ acts on $L_{x_{i}}$ as a primitive $m$-th root of unity $\lambda$. Thus in  \eqref{apb4.20},
$e^{i\theta_{\zeta^{u}}}= \lambda^{u}$.
For  $u\in \{1,\cdots, m-1\}$, set
\begin{align}\label{ross1.7}
   \eta_u=& e^{-2\pi \langle(1-\zeta^{-u})\wi{z}_{2,\zeta^{u}},
   \ov{\wi{z}}_{2,\zeta^{u}}\rangle},\\
S_{u}(\wi{Z}) =& \sum_{i} c_{i} \sum_{j=0}^{2N+2r+1} (p+i)^{n-\frac{j}{2}}
\cK_{j,\wi{Z}_{1,\zeta^{u}}}(\sqrt{p+i}\wi{Z}_{2,\zeta^{u}})  
\lambda^{u (p+i)}  \eta_u^{p+i},\nonumber\\
\mS_{2}=& \sum_{u=1}^{m-1} S_{u},\qquad
\mS_{1}= \sum_{i} c_{i}
\sum_{j=0}^{N+r} \pmb{b}_j(\wi{Z}) (p+i)^{n-j}. \nonumber
\end{align}
Here $Z=z+\overline{z}$, and $z=\sum_i z_i\tfrac{\partial}{\partial z_i}$,
$\overline{z}= \sum_i\overline{z}_i\tfrac{\partial}{\partial\overline{z}_i}$
when we consider them as vector fields, and
$\Big\lvert\tfrac{\partial}{\partial z_i}\Big\rvert^2=
\Big\lvert\tfrac{\partial}{\partial\overline{z}_i}\Big\rvert^2
=\dfrac{1}{2}$. Similarly for $\wi{Z}$ (and those with subscripts).

Applying \eqref{apb4.20} for $k=N+r+1$ we obtain for $|\wi{Z}|\leq\var/2$,
\begin{multline}\label{ross1.8}
   \left |B^{\orb}_{p}(\wi{Z})- \mS_{1}-\mS_{2}\right|_{\cC^{l}}\\
  \leqslant   C_{l} \, p^{n-N-r-2}\left(1 + p^{\frac{l+1}{2}}
   \left(1+\sqrt{p} d (Z,X_{\sing})\right)^M
   e^{- \sqrt{C_0\, p}\, d (Z,X_{\sing})}\right)\\
+ \sum_{i} c_{i}  (p+i)^{n-N-r-1} \left( \sum_{u=1}^{m-1}
\Big |\cK_{2N+2r+2,\wi{Z}_{1,\zeta^{u}}}(\sqrt{p+i}\wi{Z}_{2,\zeta^{u}})
\eta_{u}^{p+i}\Big|_{\cC^{l}}
+ \Big|\pmb{b}_{N+r+1}(\wi{Z})  \Big|_{\cC^{l}} \right).
\end{multline}
In what follows, we write  for simplicity $\wi{Z}_{1,\zeta^{u}}$ as
$Z_{1,u}$ and $\wi{Z}_{2,\zeta^{u}}$ as $Z_{2,u}$.
For a function $f(p,Z)$ with $p\in \N$ and $|Z|\leq \var/2$ we write
$f=\mO_{\cC^{j}}(g(p,Z))$ if the $\cC^{j}$-norm of $f$ in $Z$ can be
uniformly controlled by $C \, |g(p,Z)|$.

Note that $\cK_{j,Z_{1,u}}(Z_{2,u})$ is a polynomial in $Z_{2,u}$ with the same
parity as $j$ and $\deg \cK_{j, Z_{1,u}}\leq 3j$. Denote by
$\cK_{j, Z_{1,u}, l}$ the $l$-homogeneous part of
$\cK_{j,Z_{1,u}}$. Then $\cK_{j, Z_{1,u}, l}=0$ if $l$ and $j$
are not in the same parity or $l> 3j$. By (\ref{ross1.7}),
\begin{multline}\label{ross1.10}
S_{u}(Z) = \sum_{i} c_{i} \sum_{j=0}^{2N+2r+1} \sum_{l}
(p+i)^{n-\frac{j-l}{2}}
\cK_{j,Z_{1,u}, l }(Z_{2,u})  \lambda^{u (p+i)}  \eta_u^{p+i}\\
= \lambda^{u p}  \sum_{j=0}^{2N+2r+1}  \left\{   \Big(\sum_{l\geq j-2n}
\sum_{q=0}^{n-\frac{j-l}{2}}
+ \sum_{l< j-2n} \sum_{q=0}^{N+r}\Big)
\cK_{j,Z_{1,u}, l }(\sqrt{p}Z_{2,u})\right.\\
\times \left. \, p^{n-\frac{j}{2} -q}
\begin{pmatrix} n-\frac{j-l}{2} \\ q \end{pmatrix}
\sum_{i} c_{i} i^{q} \lambda^{u i} \eta_u^{p+i}\right\} +
\mO_{\cC^{2r}}(p^{n-N-1}).
\end{multline}
Here we used $(p+i)^{\gamma}= \sum_{q=0}^{N+r} p^{\gamma-q}
\begin{pmatrix} \gamma \\ q \end{pmatrix} i^{q} +
\mO(p^{\gamma-N-r-1})$  for $\gamma <0$ and
the following relations for  $r',r''\in \N$, $r''\leq l$,
\begin{align}\label{ross1.11}\begin{split}
&\cK_{j,Z_{1,u}, l }(\sqrt{p}Z_{2,u}) \eta_{u}^{p} =
\mO_{\cC^{r'}}(p^{\frac{r'}{2}}\eta_{u}^{p/2}),\\
&  \cK_{j,Z_{1,u}, l }(\sqrt{p}Z_{2,u})  =
\mO_{\cC^{r''}}(p^{\frac{l}{2}} |Z_{2,u}|^{l-r''}).
\end{split}\end{align}
In order to prove (\ref{ross1.4}) it is sufficient to show that
for $0\leq l \leq N+r$,  $r'\leq 2r$,
\begin{align}\label{ross1.12}
w_{l,p}:=\sum_{i} c_{i} i^{l} \lambda^{u i} \eta_u^{p+i}
= \mO_{\cC^{r'}}(p^{l-N-r-1 +\frac{r'}{2} }  \eta_u^{p/2}) .
\end{align}
In fact, we will prove that
$w_{l,p}= \mO_{\cC^{r'}}(p^{l-N-r-1 +\frac{r'}{2} }
\eta_u^{(\frac{3}{4}- \frac{r'}{8r})p}) $ for $r'\leq 2r$.

Since $dw_{l,p}= \frac{d\eta_u}{\eta_u} (p w_{l,p}+ w_{l+1,p})$, and
$\frac{d\eta_u}{\eta_u}$ has a term $z_{2,u}$ or $\ov{z}_{2,u}$ 
which can be absorbed by
$\eta_u^{ \frac{1}{8r}p} $ to get a factor $p^{-1/2}$, we see by
induction that it is sufficient to prove $w_{l,p}=
\mO_{\cC^{0}}(p^{l-N-r-1}  \eta_u^{\frac{3}{4} p})$. To this end,
write
\begin{align}\label{ross1.15}
   w_{l,p}= \left[\frac{\sum_{i} c_{i} i^{l} \lambda^{u i} \eta_u^{i} }
   {(\eta_u-1)^{N+r-l+1}} \right](\eta_u-1)^{N+r-l+1}\eta_u^{p}.
\end{align}
Since $\lambda$ is a primitive $m$-th root of unity, $\lambda^{u}\neq
1$ if $u\in \{1, \cdots,m-1\}$.
From \cite[Lemma 3.5]{RT09a}, under
the condition (\ref{ross1.1}), the function $\eta\to \sum_{i} c_{i}
i^{l} \lambda^{u i} \eta^{i}$ has a root of order $N+r-l+1$ at
$\eta=1$ and so the term in square brackets is bounded.

For $|z_{2,u}|\leq \var$, we have by \eqref{ross1.7},
   \begin{align}\label{ross1.17}
   |\eta_u-1 |\leq C |z_{2,u}|^2.
   \end{align}
   By using \eqref{ross1.7} and \eqref{ross1.17} and the fact that $[0,\infty)\ni x\mapsto x^{s}e^{-x}$ is bounded
   for any $s\geq 0$, we get
   \begin{align}\label{ross1.18}
(\eta_u-1)^{s}\eta_u^{p/4}  = \mO(p^{-s}) \quad \text{ for }  s\geq 1.
   \end{align}

Thus, $w_{l,p}= \mO_{\cC^{0}}(p^{l-N-r-1} \eta_u^{\frac{3}{4} p})$
and (\ref{ross1.12}) follows.

Back in \eqref{ross1.10}, for $q> N+r$, the
corresponding contribution is certainly $p^{n-\frac{j}{2} -q}  \cdot
\mO_{\cC^{2r}}(p^{r}  \eta_u^{p/2}) =  \mO_{\cC^{2r}}(p^{n-N-1}
\eta_u^{p/2})$, by \eqref{ross1.11}. 
On the other hand, if $0\leq q\leq N+r$, then,
by \eqref{ross1.11} and \eqref{ross1.12},
the corresponding contribution is $p^{n-\frac{j}{2} -q}  \cdot
\mO_{\cC^{2r}}(p^{q-N-r-1 + r } (1+\sqrt{p}|Z_{2,u}|)^{6N+6r+3}\eta_u^{p/2}) =
\mO_{\cC^{2r}}(p^{n-N-1}  \eta_u^{p/4}) $ again.
Thus  $S_{u}= \mO_{\cC^{2r}}(p^{n-N-1}) $.

From  \eqref{ross1.7} and the above
argument, $\mS_{2}= \mO_{\cC^{2r}}(p^{n-N-1}) $. Combining with
\eqref{ross1.7}, \eqref{ross1.8} and \eqref{ross1.11},
we get \eqref{ross1.4} and \eqref{ross1.5}.
\end{proof}

\end{document}